\documentclass[12pt]{amsart}
\usepackage{calc,pifont}
\usepackage{graphicx,adjustbox}
\usepackage{mathrsfs}
\usepackage{amssymb,amsbsy}
\usepackage{amscd}
\usepackage{enumerate}
\usepackage{setspace}
\usepackage{tikz}
\usepackage{tikz-cd}
\usepackage{bm}
\usepackage{cite}
 \usepackage{graphicx}
\usepackage{hyperref}
\allowdisplaybreaks
\usepackage{cite}

\usepackage{kpfonts}
\usepackage{stackengine}
\usepackage{calc}
\newlength\shlength
\newcommand\xshlongvec[2][0]{\setlength\shlength{#1pt}%
  \stackengine{-5.6pt}{$#2$}{\smash{$\kern\shlength%
    \stackengine{7.55pt}{$\mathchar"017E$}%
      {\rule{\widthof{$#2$}}{.57pt}\kern.4pt}{O}{r}{F}{F}{L}\kern-\shlength$}}%
      {O}{c}{F}{T}{S}}
  
\usepackage[letterpaper]{geometry}
\theoremstyle{definition}
\newtheorem{theorem}{Theorem}[section]
\newtheorem{lemma}[theorem]{Lemma}
\newtheorem{proposition}[theorem]{Proposition}
\newtheorem{definition}[theorem]{Definition}
\newtheorem{corollary}[theorem]{Corollary}

\newtheorem{example}[theorem]{Example}

\newcommand{\N}{{\mathbb{N}}}

\title{Generalized cell maps }

\author{Carlos Islas}
\address{Universidad Aut\'onoma de la Ciudad de M\'exico}
\email{carlos.islas@uacm.edu.mx}
\author{Benjam\'in A.~Itz\'a-Ortiz} 
\address{\'Area Acad\'emica de Matem\'aticas y F\'isica, Universidad Aut\'onoma del Estado de Hidalgo, Pachuca, Hidalgo, Mexico}
\email{itza@uaeh.edu.mx}

\author{Roc\'io Leonel}
\address{Universidad Nacional Rosario Castellanos, Mexico City, Mexico}
\email{rocio.leonel@rcastellanos.cdmx.gob.mx}

\begin{document}

\begin{abstract}
In \cite{leonel}  the notion of a g-cell structure was introduced as a generalization of the construction proposed by Debski and Tymchatyn to realize a certain class of topological spaces as quotient spaces of inverse limits. 
In \cite{DT}, cell maps are defined between cell structures and showed
that a cell map between two complete cell structures induces a continuous
function between the spaces determined by the cell structures. In this paper, 
for g-cell structures, we introduce the notions of weak g-cell maps and g-cell maps. We give some conditions for weak g-cell and g-cell mappings between g-cell structures to induce continuous mappings between their corresponding spaces. 
\end{abstract}

\maketitle

\section*{Introduction}

In \cite{DT}  cell maps between cell structures were defined. They showed that cell maps
between cell structures determined continuous functions between the
topological spaces represented by those cell structures. In \cite{DT2018},
 the class of topologically complete spaces admitting a cell structure is extended and also showed
that topologically complete spaces and their continuous function are
determined by discrete approximations.
In \cite{leonel} the notion of cell structure was modified to include cellular graphs that are not necessarily discrete, resulting in a larger class of topological spaces including non-regular spaces.

In this work, we define weak g-cell maps, showing the existence of this functions for cell structures and, following the definition of cell map for cell structures given in \cite{DT2018}, we define g-cell maps between g-cell structures. We show that given a cell map between cell structures induces a weak g-cell map. And also that a g-cell map induces a weak g-cell map. Despite the similarity in the definitions of the functions involved, the continuity of the induced functions is no longer taken for granted, as g-cell structures are not necessarily discrete. However, we are able to provide some conditions to obtain the continuity of the induced functions.  

While the decomposition of a topological space in smaller discrete components is an interesting endeavor, it seemed to be restricted to complete metric spaces. With the notion of g-cell structures a larger class of topological spaces is allowed. However, the well behaviour of maps between g-cell structures was missing, which this paper aboards***.

\section{Preliminaries}
The set of natural numbers is as usual ${%
\mathbb{N}}=\{1,2,3,\ldots \}$ and $\Delta_X$ is the diagonal in $X\times X$, where $X$ is a topological space.

We will follow the terminology in \cite{leonel}. We say that an ordered pair $(G,r)$ is a cellular graph if $G$ is a
non-empty topological space and $r\subset G\times G$ is a reflexive and
symmetric relation on $G$. Suppose that $\{(G_{n},r_{n})\}_{n\in {\mathbb{N}}}$ is
a sequence of cellular graphs and let $\{g_{n}^{n+1}\}_{n\in {\mathbb{N}}}$
be a family of continuous functions $g_{n}^{n+1}\colon G_{n+1}\rightarrow
G_{n}$, called bonding maps, such that
$g_{n}^{n}$ is the identity in $G_{n}$, $g_{n}^{l}=g_{n}^{m}\circ g_{m}^{l}$ for $n<m<l$ and the bonding maps send edges to edges, that is, $ (g_{n}^{n+1}(x),g_{n}^{n+1}(y))\in r_{n}$ whenever $(x,y)\in r_{n+1}$. We will then say that $\left\{ \left( G_{n},r_{n}\right)
,g_{n}^{n+1}\right\} _{n\in {\mathbb{N}}}$ is an inverse sequence of
cellular graphs. 

If $\{(G_{n},r_{n})\}_{n\in {\mathbb{N}}}$ is a sequence of cellular graphs, the space $\bigcup_{n\in \N}G_n$ is endorsed with the disconnected sum topology, that is, $O\subset \bigcup_{n\in \N} G_n$ is open if $O\cap G_n$ is open in $G_n$ for all $n\in \N$. If $(G_n,r_n)$ and $(H_n,s_n)$ are sequences of cellular graphs, a map $f:~\underset{i\in 
\mathbb{N}
}{\bigcup }G_{i}\rightarrow \underset{i\in 
\mathbb{N}
}{\bigcup }H_{i}$ is said to be upper semi-continuous if  
the set $$\left\{ x\in {\bigcup_{n\in\N} }G_{n} :f\left( x\right) \subset W_i \right\} $$ is open. On the other hand, we say that $f$ is lower semi-continuous if the set 
\[
\left\{x\in {\bigcup_{n\in\N} }G_{n} \colon f(x)\cap W_i\neq\emptyset\right\}
\]
is open for every open $W_i\subset H_i$ \cite{Kuratowski32},\cite[P. 63]{E},\cite{Brown}. We say that $f$ is continuous if it is both upper and lower semi-continuous.

The inverse limit of an inverse sequence of cellular graphs 
$\left\{ \left( G_{n},r_{n}\right) ,g_{n}^{n+1}\right\} _{n\in {\mathbb{N}}}$%
, denoted by $G\sb\infty $ or $\displaystyle\lim_{\longleftarrow
}\{G_{n},g_{n}^{n+1}\}$, is the subspace $\left\{ \bar{x}=(x_{n})_{n\in {%
\mathbb{N}}}\colon \forall j\geq i,\,\,g_{i}^{j}(x_{j})=x_{i}\right\} $ of $%
\prod_{n\in {\mathbb{N}}}G_{n}$ endowed with the product topology. Elements
in $G\sb\infty $ are called threads and, by abuse of notation, the mapping $p_{i}\colon G\sb\infty
\rightarrow G_{i}$ denotes the restriction of the standard projection map $%
p_{i}\colon \prod_{n\in {\mathbb{N}}}G_{n}\rightarrow G_{i}$, for each $i\in 
\mathbb{N}
$.

If $\left( G,r\right) $ is a cellular graph and $u\in G$, we define the
neighborhood of $u$, denoted by $B\left( u,r\right) $, the following set: $%
\left\{ v\in G:\left( u,v\right) \in r\right\} $, and for $A\subset G$, we
denoted $B\left( A,r\right) =\underset{a\in A}{\bigcup }~B\left( a,r\right) .$ A neighborhood of $u$ with radius $2r$ is defined as $$%
B(u,2r) =\left\{ a\in G:\text{ there is }b\in
B\left( u,r\right) \text{ such that }a\in B\left( b,r\right)
\right\}. $$
Similarly, a neighborhood of $u$ with radius $3r$ is defined as
$$%
B(u,3r) =\left\{ a\in G:\text{ there is }b\in
B\left( u,2r\right) \text{ such that }a\in B\left( b,r\right)
\right\}. $$

An inverse sequence of cellular graphs $\left\{ \left( G_{n},r_{n}\right)
,g_{n}^{n+1}\right\} _{n\in \N }$, with $G_n$ discrete, is a cell structure if it satisfies that for
each $\overline{x}\in G_{\infty }$ and each natural number $i$, there exists
an integer $j\geq i$ such that $g_{i}^{j}\left( B(x_j,2r_j) \right) \subset B\left( x_{i},r_{i}\right) $,  and for each $\overline{x}\in G_{\infty }$ and each natural
number $i$, there exists an integer $j\geq i$ such that $g_{i}^{j}\left(
B\left( x_{j},r_{j}\right) \right) $ is finite. 

Let $\left\{ \left( G_{n},r_{n}\right)
,g_{n}^{n+1}\right\} _{n\in \N }$ be 
 an inverse sequence of cellular graphs. Notice that the $G_n$ are not necessarily discrete. We define the natural relation $r$ on $G\sb\infty $ by $r=\{(\bar{x},%
\bar{y})\in G\sb\infty \times G\sb\infty \colon \,(x_{n},y_{n})\in r_{n}$, $%
\forall n\in \mathbb{N}\}$. We say that $\left\{
\left( G_{n},r_{n}\right) ,g_{n}^{n+1}\right\} _{n\in {\mathbb{N}}}$ is a
generalized cell structure, or a g-cell structure for short, if $r$ is an equivalent relation.


Let $\left\{ \left( G_{n,}r_{n,}\right) ,g_{n}^{n+1}\right\} _{n\in 
\mathbb{N}
}$ and $\left\{ \left( H_{n,}r_{n,}^{\prime }\right) ,h_{n}^{n+1}\right\}
_{n\in 
\mathbb{N}
}$ be g-cell structures, with $G_{\infty }=\underset{\leftarrow }{~\lim }%
\left\{ \left( G_{n,}r_{n,}\right) ,g_{n}^{n+1}\right\} _{n\in 
\mathbb{N}
}$, $H_{\infty }=\underset{\leftarrow }{~\lim }\left\{ \left(
H_{n,}r_{n,}^{\prime }\right) ,h_{n}^{n+1}\right\} _{n\in 
\mathbb{N}
}$. Consider $r$ and $r^{\prime }$ the equivalence relation in $G_{\infty
} $ and $H_{\infty }$ respectively, and denoted by $G^{\ast }=G_{\infty }/r$
and $H^{\ast }=H_{\infty }/r^{\prime }.$ For $\overline{x}\in G\sb\infty$, we denote its corresponding class by $\pi(\overline{x})=[\overline{x}]\in G\sp\ast$ (respectively, $\pi\sp\prime\colon H\sb\infty\to H\sp\ast$ the quotient map, and  $\pi\sp\prime(\overline{y})=[\overline{y}]$). If $x_{i}\in G_{i}$ and   $U$ is a neighborhood
of $B\left( x_{i},r_{i}\right) $, we denote by $A_{x_{i},U}$ the following
set:
$$\left\{ \left[ \overline{z}\right] \in G^{\ast }:z_{i}\in U\text{ and there
exist }j>i\text{ such that }\left( z_{j},w_{j}\right) \notin r_{j},~\forall 
\overline{w}\in p_{i}^{-1}\left( B\left( U,r_{i}\right) \diagdown U\right)
\right\} .$$

Let $\left\{ \left( G_{n},r_{n}\right) ,g_{n}^{n+1}\right\} _{n\in \N }$ be a
cell structure. We say that two cells $a\in G_{m}$ and $b\in G_{n}$ are $close$ if $\left(
g_{k}^{m}\left( a\right) ,g_{k}^{n}\left( b\right) \right) \in r_{k}~$with $%
k=\min \left\{ n,m\right\}$.

A sequence $\left\{ u_{j}\right\} _{j\in 
\mathbb{N}
}$ of cells in $\underset{i\in 
\mathbb{N}
}{\bigcup }G_{i}$ is said to be a $Cauchy$ $sequence$ if and only if

\begin{enumerate} 
\item[$a)$] $\underset{j\rightarrow \infty }{\lim}\deg \left( u_{j}\right)
=\infty ,$

\item[$b)$] $u_{i}$ and $u_{j}$ are close for all $i$ and $j~$ sufficiently
large, that is, there is $N\in 
\mathbb{N}
$ such that for $m,n>N$, $u_{m}$ and $u_{n}$ are close.
\end{enumerate}

The sequence $\left\{ u_{j}\right\} _{j\in 
\mathbb{N}
}$ is said to be $converge~to~a~thread~\overline{x}\in G_{\infty }$ if $\underset{%
j\rightarrow \infty }{\lim}\deg \left( u_{j}\right) =\infty $ and $x_{i}$,$%
u_{j}$ are close for all $i$ and for all $j$ sufficiently large, that is,  there is $N$ such that if $i,j>N$ then $x_i$ is close to $u_j$.

If each Cauchy sequence of cells converges we say that $\left\{ \left( G_{n},r_{n}\right) ,g_{n}^{n+1}\right\} _{n\in \N }$ is a complete
cell structure.

\section{Weak g-cell maps}

In \cite{DT}, Debski and Tymchatyn defined cell maps between cell structures and show
that a cell map between two complete cell structures induces a continuous
function between the spaces determined by the cell structures. In this way, in this section, when $\left\{\left(G_n,r_n\right),g_{n}^{n+1}\right\}_{n\in\N}$ and $\left\{\left(H_n,s_n\right),h_{n}^{n+1}\right\}_{n\in\N}$ are g-cell structures, we define a function between their inverse limits that, under certain conditions, induces a natural continuous function between $G^*$ and $H^*$.

\begin{definition}
Let $\left\{\left(G_n,r_n\right),g_{n}^{n+1}\right\}_{n\in\N}$ and $\left\{\left(H_n,s_n\right),h_{n}^{n+1}\right\}_{n\in\N}$ be inverse sequences of cellular graphs.
A weak g-cell map is a function 
\[
f\colon G\sb\infty\to H\sb\infty
\]
which satisfies the following property: If $(\overline{x},\overline{y})\in r$, then $(f(\overline{x}),f(\overline{y}))\in s$.
\end{definition}

Since a cell structure is a g-cell structure, in the following result we show that given a cell map between cell structures (see \cite{DT}), induces a weak g-cell map. The idea of the proof is contained in Lemma~7.2 of \cite{DT}, we include it here for completeness.

\begin{proposition}\label{inducedgcellmap}
Let $\left\{ \left( G_{n},r_{n}\right) ,g_{n}^{n+1}\right\} _{n\in 
\mathbb{N}
}$ be a cell structure and let $\left\{ \left( H_{n},s_{n}\right) ,h_{n}^{n+1}\right\} _{n\in 
\mathbb{N}
}$ be a complete cell structrures.
If $f\colon\underset{i\in 
\mathbb{N}
}{\bigcup }G_{i}\rightarrow \underset{i\in 
\mathbb{N}
}{\bigcup }H_{i} $ is a cell map in the sense of \cite{DT}, then it induces a weak g-cell map $\overline{f}\colon G\sb\infty\to H\sb\infty$ given by $\overline{f}(\overline{x})=\overline{y}$ where the sequence $\{f(x_n)\}_{n\in\N}$ converges to $\overline{y}$.
\end{proposition}
\begin{proof}
Since $\overline{x}=(x_n)_{n\in\N}\in G\sb\infty$  is Cauchy, ${f}(\overline{x})$ is also Cauchy and since $H$ es complete, there is $\overline{y}\in H\sb\infty$ such that $(f(x_n))\to\overline{y}$. Define $\overline{f}(\overline{x})=\overline{y}$. Let $\overline{u},\overline{v}\in G\sb\infty$ such that $(\overline{u},\overline{v})\in r$.
Assume that $f(\overline{u})=\overline{w}$ and $f(\overline{v})=\overline{z}$. We want to show that $(\overline{w},\overline{z})\in s$, equivalently, that $(w_i,z_i)\in s_i$ for all $i\in \N$.

Let $i\in \N$.  
Using the convergence of sequences, there exists $i\leq j\in\N$ large enough such that $f(u_{j})$ and $w_{j}$ are close, and $f(v_{j})$ and $z_{j}$ are close. On the other hand, since $f$ takes close cells to close cells, it also follows that  $f(u_{j})$ and $f(v_{j})$ are close. Then, for sufficiently large $N$, we have that $z_{{j+N}}$ belongs to  $B(w_{{j+N}},3s_{j+N})$. By the definition of cell structure in \cite{DT}, 
$(w_{j+N},z_{{j+N}})\in s_{j+N}$. Thus $(w_{i},z_{i})\in s_{i}$, as wanted.  \end{proof}

\begin{theorem}\label{theoweakgcell}
Let $f\colon G\sb\infty\to H\sb\infty$ be a {weak} g-cell map. Then $f$ induces a function $\widehat f\colon G\sp\ast\to H\sp\ast$ given by $\widehat f([\overline{x}])=[f(\overline{x})].$ 
If in addition $f$ is continuous,   then $\widehat f$ is continuous.
\end{theorem}
\begin{proof}
If $(\overline{x},\overline{y})\in r$,  then $(f(\overline{x}),f(\overline{y}))\in s$. Thus \[ \widehat{f} ([\overline{x}])=[f(\overline{x})]=[f(\overline{y})]=\widehat{f}([\overline{y}]),\]as wanted.

If $f$ is continuous, since $\pi\sp\prime\circ f=\widehat{f}\circ\pi$, it follows that for an open set  $O\subset H\sp\ast$ we have that $$\pi^{-1}(\widehat{f}^{-1}(O))=(\widehat{f}\circ\pi)^{-1}(O)=(\pi\sp\prime\circ f)^{-1}(O)$$
which is open since $\pi\sp\prime\circ f$ is continuous. Hence $\widehat{f}^{-1}(O)$ is open, as wanted.
Thus $\widehat{f}$ is continous.
\end{proof}

By Lemma~7.2 in \cite{DT}, given a cell map $ f\colon\underset{i\in 
\mathbb{N}
}{\bigcup }G_{i}\rightarrow \underset{i\in 
\mathbb{N}
}{\bigcup }H_{i} $, it induces a function $f':G\sp\ast\to H\sp\ast$ given by $[\overline{x}]\mapsto [f(\overline{x})]$. By Proposition~\ref{inducedgcellmap}, $f$ induced a weak g-cell map $\overline{f}\colon G\sb\infty\to H\sb\infty$ and, by Theorem~\ref{theoweakgcell} $\overline f$ induces a function $\widehat{\overline f}\colon G\sp\ast\to H\sp\ast$ which coincides with $f'$.

\begin{example}\label{Ex:Fcont}

Let $\left\{ \left( G_{n},r_{n}\right) ,g_{n}^{n+1}\right\} _{n\in
\mathbb{N}
}$ and $\left\{ \left( H_{n},s_{n}\right) ,h_{n}^{n+1}\right\} _{n\in
\mathbb{N}
}$ be the following g-cell structures: $G_{n}=\left[ 0,1\right] $, $%
g_{n}^{n+1}=Id$ and $r_{n}=\Delta_{G_{n}}$; while $H_{n}=\left[ 0,1%
\right] \times \left\{ 0\right\} \cup \left\{ 1/2 \right\} \times  [ 0,1] $, 
 $h_{n}= Id$ and $s_{n}=\Delta_{H_{n}}\bigcup \left\{ \Bigl(\left( \frac{1}{2}, y \right), \left( \frac{1}{2}, 0 \right)\Bigr)
:y\in \left[ 0,1\right] \right\} \bigcup \left\{ \Bigl(\left( \frac{1}{2}, 0 \right) , \left( \frac{1}{2}, y \right)\Bigr)
:y\in \left[ 0,1\right] \right\}$. 
Then 
$G_{\infty}=\left[0,1\right]$, 
$H_{\infty}=H_1$ and 
$G^{\ast }=H^{\ast }= [0,1] $. Let $F\colon G\sp\ast\to H\sp\ast$ be the identity function. We may choose $f:G_{\infty}\to H_{\infty}$ as $f(x)=\left\{ \begin{array}{cc}
(x,0) & \text{if } x\neq \frac{1}{2}  \\
(1/2,1) & \text{if } x= \frac{1}{2} 
\end{array} 
\right.$.
Then $f$ is a weak g-cell map and $\widehat{f}=F$, $F$ is continuous but $f$ is not continuous.
\end{example}

In the following we give conditions for $f$ to be continuous when $F$ is continuous. It is analogous to Proposition 7.5 in \cite{DT}. 
\begin{theorem}\label{liftF}
Let $\left\{ \left( G_{n},r_{n}\right) ,g_{n}^{n+1}\right\} _{n\in 
\mathbb{N}
}$  and   $\left\{ \left( H_{n},s_{n}\right) ,h_{n}^{n+1}\right\} _{n\in 
\mathbb{N}
}$ be  g-cell structures. If $F\colon G\sp\ast\to H\sp\ast$ is a  function, then there exists a {weak} g-cell map $f\colon G\sb\infty\to H\sb\infty$ such that $F=\widehat{f}$, where $\widehat{f}$ is the induced map defined in Theorem~\ref{theoweakgcell}. Suppose that in addition $F$ is continuous, and one of the following two conditions is satisfied:
\begin{enumerate}
\item The quotient map
$\pi\sp\prime$ is open and 
$(\pi\sp\prime)^{-1}\left(\pi\sp\prime(V)\right)\subset V$ for every open $V\subset H\sp\ast.$
\item For each $G_{n}$ satisfies that given $x\in G_{n}$ and
for every open neighborhood $U$ of $B\left( x,r_{n}\right) $, there
exists an open neighborhood $O$ of $x$ such that $B\left( O,r_{n}\right)
\subset U$.
\end{enumerate}
Then $f$ is also continuous.
\end{theorem}
\begin{proof}
Let $\overline{x}\in G\sb\infty$. Then, since $F([\overline{x}])\in H\sp\ast$, there is $\overline{y}\in H\sb\infty$ such that $[\overline{y}]=F([\overline{x}])\in H\sp\ast$. Define $f(\overline{x})=\overline{y}$. Hence $\widehat{f}([\overline{x}])=[f(\overline{x})]=[\overline{y}]=F([\overline{x}])$. Now suppose that $(\overline{x},\overline{x\sp\prime})\in r$, then $[\overline{x}]=[\overline{x\sp\prime}]$. Assume that  $f\left(\overline{x\sp\prime}\right)=\overline{y\sp\prime}$, as $[\overline{y}]=F([\overline{x}])=F([\overline{x\sp\prime}])=[\overline{y\sp\prime}]$, then $(\overline{y},\overline{y\sp\prime})\in s$. Hence $f$ is a weak g-cell map.

Now, assume that $F$ is continuous. 

Suppose $\pi\sp\prime$ is open and $(\pi\sp\prime)^{-1}(\pi\sp\prime(V))\subset V$ for every open $V\subset H\sp\ast$. Let $\overline{x}\in G\sb\infty$, $\overline{y}=f(\overline{x})$ and $\overline{y}\in V\subset H\sb\infty$ open. Since $\pi\sp\prime$ is open, then $U=\pi^{-1}(F^{-1}(\pi\sp\prime(V)))\subset G\sb\infty$ is open and contains $\overline{x}$.  Since by hypothesis 
$(\pi\sp\prime)^{-1}(\pi\sp\prime(V))\subset V$, we get that $f(U)\subset V$. Thus, $f$ is a continuous function. 

On the other hand, suppose each $G_{n}$ satisfies that given $x\in G_{n}$ and
for every open neighborhood $U$ of $B\left( x,r_{n}\right) $, there
exists an open neighborhood $O$ of $x$ such that $B\left( O,r_{n}\right)
\subset U$. Let $\overline{x}\in G_{\infty }$, $\overline{y}=f\left( \overline{x}\right) 
$, $V$ a basic open set of $H_{\infty }$ that contains $\overline{y}$. Then $%
V=p_{i}^{-1}\left( W\right) $ for some $W$ open subset of $H_{i}.$ By Lemma~1 of \cite{leonel}, the set $A_{y_{i},{W}}$ is open in $H^{\ast }$ and this satisfies $\pi ^{\prime
}\left( \overline{y}\right) \in A_{y_{i},W}.$ Let $U=\pi ^{-1}\left( F^{-1}\left( A_{y_{i},W}\right) \right) $, then $U$
is an open subset of $G_{\infty }$.
Since $\pi 
{'}%
\left( \overline{y}\right) =\left[ \overline{y}\right] =F\left( \left[ 
\overline{x}\right] \right)$, so $\pi(\overline{x})=[\overline{x}]\in F^{-1}\left(
A_{y_{i},W}\right) $, this implies that $\overline{x}\in U.$ Now, if $\overline{z}\in U$, let $f\left( \overline{z}\right) =\overline{a}$. This implies that $%
\left[ \overline{a}\right] =F\left( \left[ \overline{z}\right] \right) $ and $[\overline{z}]=\pi \left( \overline{z}\right) \in F^{-1}\left( A_{y_{i},W}\right) $; therefore $F\left( %
\left[ \overline{z}\right] \right) \in A_{y_{i},W}.$ Hence  $%
\left[ \overline{a}\right] \in A_{y_{i},W}$, then $a_{i}\in W$ and we get
that $f\left( U\right) \subset V$. Thus, $f$ is a continuous function.
\end{proof}

Of course, since condition (2) in the hypothesis of Theorem~\ref{liftF} is satisfied in Example~\ref{Ex:Fcont}, it is possible to construct a function $f$ which is continuous. This is achieved by choosing $f(x)=(x,0)$ for all $x\in \left[0,1\right]$.

The following example shows that even when $F$ is continous, it is not possible to choose an induced continuous function $f$.

\begin{example}

Let $\left\{ \left( G_{n},r_{n}\right) ,g_{n}^{n+1}\right\} _{n\in
\mathbb{N}
}$ and $\left\{ \left( H_{n},s_{n}\right) ,h_{n}^{n+1}\right\} _{n\in
\mathbb{N}
}$ be the following g-cell structures: $G_{n}=\left[ 0,1\right] $, $%
g_{n}^{n+1}=Id$, and $r_{n}=\triangle _{G_{n}}$; while $H_{n}=\left[ 0,1%
\right] $, 

$h_{n}^{n+1}\left( x\right) =\left\{
\begin{array}{cc}
4x & \text{if }x\in \left[ 0,\frac{1}{4}\right]  \\
\frac{3}{2}-2x & \text{if }x\in \left[ \frac{1}{4},\frac{1}{2}\right]  \\
x & \text{if }x\in \left[ \frac{1}{2},1\right]
\end{array}%
\right. $. Then $H_{\infty }$ is homeomorphic to $\left\{ 0\right\} \times %
\left[ -1,1\right] \cup \left\{ \left( x,\sin (1/x)\right) :x\in
(0,1]\right\} $, the topologist's sine curve, which is not locally connected. 

Now,
consider $s_{n}=\Delta _{H_{n}}\cup \left\{ \left( x,\frac{1}{2}\right)
:x\in \left[ \frac{1}{2},1\right] \right\} \cup \left\{ \left( \frac{1}{2}%
,~x\right) :x\in \left[ \frac{1}{2},1\right] \right\} $.

Then
$H^{\ast }$ and $G^{\ast }$ are
homeomorphic to an arc. 
Let $F\colon G\sp\ast\to H\sp\ast$ be the identity function. Then $f$ is a weak g-cell map and $\widehat{f}=F$, $F$ is continuous but since 
$f$ takes values in $\{0\}\times[-1,1]$ and in 
$\left\{ \left( x,\sin (1/x)\right) :x\in
(0,1]\right\}$ then
it  cannot be continuous. 

 \end{example}








\section{g-cell maps}

In this section, we introduce g-cell maps, a type of set-valued function defined between the unions of factor spaces associated with two given g-cell structures. 

In contrast to weak g-cell maps, which are defined on the inverse limits of g-cell structures, g-cell maps are set-valued functions between the unions of factor spaces. We show these g-cell maps induce weak g-cell maps. Furthermore, we establish the necessary conditions under which a g-cell map induces a continuous weak g-cell map.
We follow the idea of cell maps given in \cite{DT2018} to define these g-cell maps for g-cell structures.

\begin{definition}\label{def:gcellmap}
Let $\left\{\left(G_n,r_n\right),g_{n}^{n+1}\right\}_{n\in\N}$ and $\left\{\left(H_n,s_n\right),h_{n}^{n+1}\right\}_{n\in\N}$ be inverse sequences of cellular graphs.
A g-cell map is a set-valued function $f$
$:~\underset{i\in 
\mathbb{N}
}{\bigcup }G_{i}\rightarrow \underset{i\in 
\mathbb{N}
}{\bigcup }H_{i}$  which satisfies the following properties:

\begin{enumerate}
\item $h_{i}^{j}\left( f\left( x\right) \cap H_{j}\right) \subset f\left(
x\right) \cap H_{i}$ for $i\leq j$,

\item if $i\leq j$, $k\in 
\mathbb{N}
$ and $x\in G_{j}$ with $f\left( g_{i}^{j}\left( x\right) \right) \cap
H_{k}\neq \emptyset $, then $f\left( x\right) \cap H_{k}\neq
\emptyset $ and $f\left( x\right) \cap H_{k}\subset f\left(
g_{i}^{j}\left( x\right) \right) $,

\item if $a\in f\left( x\right) \cap H_{k}$, $b\in f\left(
y\right) \cap H_{k}$, where $x,y\in G_{i}$ are such that 
$\left( x,y\right) \in r_{i}$, then $\left( a,b\right)
\in s_{k}$,

\item for each $\overline{x}\in G_{\infty }$ and $k\in\N$ 
there exists $i\in 
\mathbb{N}
$ such that $f\left( x_{i}\right) \cap H_{k}$ is a Hausdorff, compact, nonempty set.
\end{enumerate}

\end{definition}

The following lemma provides conditions for the function $f$ to preserve the closeness relationship between elements of the inverse sequence.

\begin{lemma}\label{closecells}
    Let $\left\{\left(G_n,r_n\right),g_{n}^{n+1}\right\}_{n\in\N}$ and $\left\{\left(H_n,s_n\right),h_{n}^{n+1}\right\}_{n\in\N}$ be inverse sequences of cellular graphs.
A g-cell map $f$
$:~\underset{i\in 
\mathbb{N}
}{\bigcup }G_{i}\rightarrow \underset{i\in 
\mathbb{N}
}{\bigcup }H_{i}$ satisfies the following: 

\begin{enumerate}
    \item[(1a)]  Let $\overline{x}$ and $\overline{y} \in G_{\infty}$  such that  $(x_{i},y_{i})\in r_{i}$ for some $i\in \N$. If $a\in f(x_i)$  and $b\in f(y_i)$, then $a$ and $b$ are close.
    \item[(1b)]  Let $\overline{x}$ and $\overline{y} \in G_\infty$  such that  $x_i$ is close to $y_j$. If $a\in f(x_i)$ and $b\in f(y_j)$, then $a$ and $b$ are close.
\end{enumerate}
\end{lemma}
\begin{proof}
   \begin{enumerate}
       \item[(1a)] 
   
   Assume $a\in f(x_i)\cap H_j$ for some $j\in N$. There are two cases. If $b\in f(y_i)\cap H_k$ and $k\geq j$. By (1), we have that $h_j^k(b)\in f(y_i)\cap H_j$. Hence by (3), $(a, h_j^k(b))\in s_j$. Similarly, if $b\in f(y_i)\cap H_k$ and $k<j$, then by (1) and (3), $(h_k^j(a),b)\in s_k$. 

    \item[(1b)] Without loss of generality, assume $i\leq j$. Then $(x_i,g_i^j(y_j))=(x_i,y_i)\in r_i$. Using (1a), we obtain that $a$ and $b$ are close, as wanted.
    \end{enumerate}
\end{proof}

In the following proposition, given maps between the respective factor spaces, a g-cellular function is naturally constructed.

\begin{proposition}

Let $\left\{ \left( G_{n},r_{n}\right) ,g_{n}^{n+1}\right\} _{n\in 
\mathbb{N}
}$ and $\left\{ \left( H_{n},s_{n}\right) ,h_{n}^{n+1}\right\} _{n\in 
\mathbb{N}
}$ be inverse sequences of cellular graphs. If for each $i\in\N$ there is a function $f_{i}:G_{i}\rightarrow H_{i}$ such that $f_{i}\circ
g_{i}^{i+1}=h_{i}^{i+1}\circ f_{i+1}$ and $(f_i(x),f_i(y))\in s_i$ whenever $(x,y)\in r_i$, then there exists a g-cell map
\begin{equation*}
f:~\underset{i\in 
\mathbb{N}
}{\bigcup }G_{i}\rightarrow \underset{i\in 
\mathbb{N}
}{\bigcup }H_{i}
\end{equation*}
given by,
for $x\in G_i$,
\begin{equation*}
f\left( x\right) =\bigcup_{j=0}^{i} h_{i-j}^{i}(f_{i}(x))= \bigcup_{m=1}^{i} p_m^\prime\left(\langle f_i(x)\rangle
\right).
\end{equation*}

%
\end{proposition}
\begin{proof}
(1) Let $n\leq m$. Assume $x\in G_i$. If $i<m$, then $f(x)\cap H_m=\emptyset$ and we are done. Else, if $m\leq i$, then $f(x)\cap H_m=\{h_m^i(f_i(x)) \}$. Hence 
$h_n^m(f(x)\cap H_m)=\{h_n^i(f_i(x))\}=f(x)\cap H_n$.

(2) Let $i\leq j$, $k\in\N$ and $x\in G_j$ such that $f(g_i^j(x))\cap H_k\neq \emptyset$. This implies that $k\leq j$ and $f(x)\cap H_k\neq\emptyset$, by the definition of $f$.  Finally $f(x)\cap H_k=\{h_k^j(f_j(x))\}=\{f_k(g_k^j(x))\}\subset f\left(g_i^j (x)\right)$.

(3)Follows from hypothesis.

(4) Let $\overline{x}\in G\sb\infty$. If $x_i\in G_i$, then $f_i
(x_i)\in H_i$ for each $i\in \N$. This implies  $f(x_i)\cap H_i=\{ f_i(x_i)\}$.

\end{proof}
Now, we show that a g-cell map induces a weak g-cell map. 
\begin{proposition}\label{prop:gcellmap}
Let $\left\{ \left( G_{n},r_{n}\right) ,g_{n}^{n+1}\right\} _{n\in 
\mathbb{N}
}$ and $\left\{ \left( H_{n},s_{n}\right) ,h_{n}^{n+1}\right\} _{n\in 
\mathbb{N}
}$ be g-cell structures. 
Let  $f:~\underset{i\in 
\mathbb{N}
}{\bigcup }G_{i}\rightarrow \underset{i\in 
\mathbb{N}
}{\bigcup }H_{i}$ be a g-cell map. Then there exists a weak g-cell map
$\overline{f}\colon G_{\infty
}\rightarrow H_{\infty }$ 
\end{proposition}
\begin{proof}



 Let $\overline{x}\in G_{\infty }$. By (4) in Definition~\ref{def:gcellmap}, for each $i\in \N$, there is $\alpha_{i}$
such that $K_i=f\left( x_{\alpha_{i}}\right) \cap H_{i}$ is Hausdorff, compact and nonempty. We may assume that $\alpha_i<\alpha_{i+1}$, since otherwise, if $\alpha_{i}\geq \alpha_{i+1}$, we may use (2) in Definition~\ref{def:gcellmap} to replace $\alpha_{i+1}$ by $\alpha_{i}+1$.
 Notice that by (1) and (2) in  Definition~\ref{def:gcellmap}, $h_{i}^{i+1}\left( f\left( x_{\alpha_{i+1}}\right) \cap H_{i+1}\right)
\subset f\left( x_{\alpha_{i+1}}\right) \cap H_{i}$ and $f\left( x_{\alpha_{i+1}}\right)
\cap H_{i}\subset f\left( g_{\alpha_{i}}^{\alpha_{i+1}}\left( x_{\alpha_{i+1}}\right) \right)
=f\left( x_{\alpha_{i}}\right)$, respectively. Then $h_{i}^{i+1}(K_{i+1})\subset K_i$ and so by Theorem~3.6 in \cite{ES}, $\displaystyle\lim_{\leftarrow}(K_i,h_{i}^{i+1}|_{K_{i+1}})\neq\emptyset$. We define $\overline{f}(\overline{x})=\overline{y}$ where $\overline{y}\in \displaystyle\lim_{\leftarrow}(K_i,h_{i}^{i+1}|_{K_{i+1}})\subset H\sb\infty$.
Let $\overline{x}$ and $\overline{x^\prime}$ in $G\sb\infty$ such that $(\overline{x},\overline{x^\prime})\in r$.  Suppose that $\overline{f}(\overline{x})=\overline{y}=(y_i)_{i=1}^{\infty}$ and $\overline{f}(\overline{x^\prime})=\overline{y^\prime}=(y_i^{\prime})_{i=1}^{\infty}\subset H\sb\infty$. For $i\in\N$, we have that $y_i\in f(x_{\alpha_i})\cap H_i$ and $y_{i}^\prime\in f(x_{\alpha_i^\prime}^{\prime})\cap H_i\subset f(x_{\alpha_i^\prime}^{\prime})$. Since $x_{\alpha_i}$ and $x_{\alpha_i^\prime}^{\prime}$ are close, an application of Lemma~\ref{closecells} yields that $y_i$ and $y_{i}^{\prime}$ are close, but since they belong to the same $H_i$, hence $(y_i,y_i^{\prime})\in s_i$. Thus $\overline{f}$ is a weak g-cell map.

\end{proof}

\begin{corollary}
Let $\left\{ \left( G_{n},r_{n}\right) ,g_{n}^{n+1}\right\} _{n\in 
\mathbb{N}
}$ and $\left\{ \left( H_{n},s_{n}\right) ,h_{n}^{n+1}\right\} _{n\in 
\mathbb{N}
}$ be g-cell structures. 
Let  $f:~\underset{i\in 
\mathbb{N}
}{\bigcup }G_{i}\rightarrow \underset{i\in 
\mathbb{N}
}{\bigcup }H_{i}$ be a g-cell map.  
Then $f$
induces a function $\widehat{f}:G^{\ast }\rightarrow H^{\ast }$ 
\end{corollary}

\begin{proof}
 By Proposition~\ref{prop:gcellmap} there is a function $\overline{f}$ induced by $f$.
 By Theorem~\ref{theoweakgcell}
 this induces $\widehat{f}$  given by $\widehat{f}\left(\left[ x\right]\right) =\left[ \overline{f}\left( x\right) \right]$.
\end{proof}

In the following example, we show that if the function $f$ in Proposition~\ref{prop:gcellmap} is continuous, then the induced map $\overline{f}$ is not necessarily continuous.

\begin{example}
Consider $\left\{ \left( G_{n},r_{n}\right) ,g_{n}^{n+1}\right\} _{n\in 
\mathbb{N}
}$ and $\left\{ \left( H_{n},s_{n}\right) ,h_{n}^{n+1}\right\} _{n\in 
\mathbb{N}
}$ two g-cell structures, where $H_n$ is compact, non-degenerated and Hausdorff. Let  $f:~\underset{i\in 
\mathbb{N}
}{\bigcup }G_{i}\rightarrow \underset{i\in 
\mathbb{N}
}{\bigcup }H_{i}$ be defined as
\begin{equation*}
f(g)=\bigcup_{k=1}^{i} H_k
\end{equation*}
if $g\in G_i$. Then it can be checked that $f$ is a g-cell map and that $f$ is both upper and lower semicontinuous. Hence $f$ is a continuous g-cell map. To construct the map $\overline{f}\colon G\sb\infty \to H\sb\infty$ of Proposition~\ref{prop:gcellmap}, given a $\overline{x}$ in $G\sb\infty$, one chooses a $\overline{y}$ in $H\sb\infty$, since $K_i=H_i$. Hence it is easy to construct a $\overline{f}$ which is not continuous.
\end{example}

In the following we give conditions for a $g$-cell map to induce a continuous function on the quotients.

\begin{theorem}
Let $\left\{ \left( G_{n},r_{n}\right) ,g_{n}^{n+1}\right\} _{n\in 
\mathbb{N}
}$ and $\left\{ \left( H_{n},s_{n}\right) ,h_{n}^{n+1}\right\} _{n\in 
\mathbb{N}
}$ be g-cell structures. 
Let  $f:~\underset{i\in 
\mathbb{N}
}{\bigcup }G_{i}\rightarrow \underset{i\in 
\mathbb{N}
}{\bigcup }H_{i}$ be an upper semi-continuous g-cell map. If for each $\overline{x}\in G\sb\infty$ and for each $i\in \N$ the set $f(x_i)\cap H_i$ is a singleton then
the weak g-cell map
$\overline{f}\colon G_{\infty
}\rightarrow H_{\infty }$, defined in Proposition~\ref{prop:gcellmap}, is continuous. Hence ${f}$ induces a continuous function $\widehat {\overline{f}}\colon G\sp\ast\to H\sp\ast$.
\end{theorem}

\begin{proof}

Let $\overline{x}=\left( x_{i}\right) _{i\in \mathbb{N}} \in G_{\infty}$, then by the defintion of $\overline{f}$, we have $%
\overline{f}\left( \overline{x}\right) =\overline{y}=\left( y_{i}\right)
_{i\in \mathbb{N}}$, with $\left\{ \overline{y}\right\} =\underleftarrow{%
\lim }\left\{ K_{i},h_{i}^{i+1}|_{K_{i+1}}\right\} $, where $K_i=f(x_i)\cap H_i$.

Let $W$ be a basic open subset that contains $\overline{y}$. Then $%
W=\pi _{i}^{-1}\left( W_{i}\right) $ for some $i\in \mathbb{N}$ and $W_i$ an open set of $G_i$.
Since $f$ is upper semi-continuous, the set $\left\{ x\in {\bigcup }G_{j} :f\left( x\right) \subset W_i \right\} =U $ is open.


Let $\overline{z}\in V=\left[ \left( U\cap G_{i}\right) \times
\prod\limits_{j\neq i}G_{j}\right] \cap G_{\infty }$. Then $\overline{f}%
\left( \overline{z}\right) \in W$. As $V$ is an open set in $G\sb\infty$ containing $\overline{x}$ and $V\subset \left(\overline{f}\right)^{-1}(W)$, we conclude that $\overline{f}$ is continuous. Using Theorem~\ref{theoweakgcell}, we obtain that $\widehat{\overline{f}}$ exists and is continuous.

\end{proof}

In the following we show that given a continuous function between the quotient spaces and under certain conditions, there exists a g-cell map such that its induced map coincides with the original function. 

\begin{theorem}\label{thm:last}
  Let $\left\{ \left( G_{n},r_{n}\right) ,g_{n}^{n+1}\right\} _{n\in 
\mathbb{N}
}$ and $\left\{ \left( H_{n},s_{n}\right) ,h_{n}^{n+1}\right\} _{n\in 
\mathbb{N}
}$ be g-cell structures. Let $F\colon G\sp\ast\to H\sp\ast$ be a continuous map such that, for each $\overline{x}\in G\sb\infty$ and $k\in \mathbb{N}$ there exists $i\in \mathbb{N}$ such that  $(\pi^{\prime})^{-1}\left(F\left(\pi \left(\langle x_i\rangle\right)\right)\right)
\cap H_k$ is a nonempty, compact and Hausdorff. Assume that $H_1$ is a simplex, that is, any two elements are related. Then there exists a g-cell map $f:~\underset{i\in 
\mathbb{N}
}{\bigcup }G_{i}\rightarrow \underset{i\in 
\mathbb{N}
}{\bigcup }H_{i}$ such that $\widehat{f}=F$. 
\end{theorem}
\begin{proof}
For $x\in G_i$, define $f(x)$ as follows;
\[
\bigcup_{j=1}^{\infty}\left\{  
h_j\left((\pi^{\prime})^{-1}\left(F\left(\pi\langle x\rangle\right)\right)\right)
\colon 
h_j\left((\pi^{\prime})^{-1}\left(F\left(\pi \left(\langle B(x,r_i)\rangle\right)\right)\right)\right)\neq\emptyset
\text{ is a simplex}\right\}
\]
We have
$f(x)\cap H_j=h_j\left((\pi^{\prime})^{-1}\left(F\left(\pi\langle x\rangle\right)\right)\right)$ by the definition of $f(x)$.
To prove (1), let $k\leq j$. Since $h_k\left((\pi^{\prime})^{-1}\left(F\left(\pi \left(\langle B(x,r_i)\rangle\right)\right)\right)\right)=h^j_k\left(h_j\left((\pi^{\prime})^{-1}\left(F\left(\pi \left(\langle B(x,r_i)\rangle\right)\right)\right)\right)\right)$ is a simplex when  $h_j\left((\pi^{\prime})^{-1}\left(F\left(\pi \left(\langle B(x,r_i)\rangle\right)\right)\right)\right)$ is a simplex and $h_k\left((\pi^{\prime})^{-1}\left(F\left(\pi \left(\langle x\rangle\right)\right)\right)\right)=h^j_k\left(h_j\left((\pi^{\prime})^{-1}\left(F\left(\pi \left(\langle x\rangle\right)\right)\right)\right)\right)$, then $h^j_k(f(x)\cap H_j)= f(x)\cap H_k$.

To prove (2), assume $l\leq i$, $j\in \N$ and $x\in G_i$
such that $f\left(
g_{l}^{i}\left( x\right) \right) \cap H_j\neq \emptyset$, then 
$h_j\left((\pi^{\prime})^{-1}\left(F\left(\pi \left(\langle B(g^i_l(x),r_l)\rangle\right)\right)\right)\right)\neq \emptyset$ is a simplex. Let  $a,b \in h_j\left((\pi^{\prime})^{-1}\left(F\left(\pi \left(\langle B(x,r_i)\rangle\right)\right)\right)\right)$, since $\langle B(x,r_i)\rangle\subset \langle B(g^i_l(x),r_l)\rangle$, $a,b \in h_j\left((\pi^{\prime})^{-1}\left(F\left(\pi \left(\langle B(g^i_l(x),r_l)\rangle\right)\right)\right)\right)$, then $(a,b)\in s_j$ and $h_j\left((\pi^{\prime})^{-1}\left(F\left(\pi \left(\langle B(x,r_i)\rangle\right)\right)\right)\right)$ is a simplex. Hence $f(x)\cap H_j\neq \emptyset$. Furthermore since 
$\langle x\rangle\subset \langle g^i_l(x)\rangle$, it follows 
 $f\left( x\right) \cap H_{j}\subset f\left(
g_{l}^{i}\left( x\right) \right). $

For (3), let $a\in f\left( x\right) \cap H_{k}$, $b\in f\left(
y\right) \cap H_{k}$, where $x,y\in G_{i}$
 such that 
$\left( x,y\right) \in r_{i}$. We show that $\left( a,b\right)
\in s_{k}$. Since $a\in h_k\left((\pi^{\prime})^{-1}\left(F\left(\pi\langle x\rangle\right)\right)\right)\subset h_k\left((\pi^{\prime})^{-1}\left(F\left(\pi \left(\langle B(y,r_i)\rangle\right)\right)\right)\right)$ and $b\in 
h_k\left((\pi^{\prime})^{-1}\left(F\left(\pi\langle y\rangle\right)\right)\right)\subset h_k\left((\pi^{\prime})^{-1}\left(F\left(\pi \left(\langle B(y,r_i)\rangle\right)\right)\right)\right)$. Therefore, both $a$ and $b$ are contained in the same simplex, so $(a,b)\in s_k$.

Finally, (4) follows by the hypothesis.

\end{proof}

We conclude this paper with a question: Is the induced function $f$ in Theorem~\ref{thm:last} continuous?

\bibliographystyle{plain}
\bibliography{mybib}{}

\end{document}